\documentclass[12pt]{article}
\usepackage{amscd,amssymb,amsmath,verbatim,amsthm, graphicx}

\usepackage{color}

\newcommand{\bS}{\mathbb S}

\newcommand{\NN}{\mathbb N}
\newcommand{\RR}{\mathbb R}
\newcommand{\ZZ}{\mathbb Z}

\newcommand{\frakp}{\mathfrak p}

\newtheorem{theorem}{Theorem}
\newtheorem{proposition}{Proposition}
\newtheorem{corollary}[proposition]{Corollary}

\newtheorem*{question}{Question}
\newtheorem*{conjecture}{Conjecture}
\theoremstyle{definition}
\newtheorem{definition}[proposition]{Definition}
\theoremstyle{remark}
\newtheorem*{remark}{Remark}

\begin{document}

\title{Spherical conic metrics and realizability of branched covers}
\author{Xuwen Zhu\\Stanford University\\xuwenzhu@stanford.edu}
\date{}

\maketitle

\begin{abstract}
Branched covers between Riemann surfaces are associated with certain combinatorial data, and Hurwitz existence problem asks whether given data satisfying those combinatorial constraints can be realized by some branched cover. We connect recent development in spherical conic metrics to this old problem, and give a new method of finding exceptional (unrealizable) branching data. As an application, we find new infinite sets of exceptional branched cover data on the Riemann sphere.  
\end{abstract}

\section{Introduction}
Branched covers between Riemann surfaces is a rich subject that has a lot of interesting features. For two Riemann surfaces $M$ and $N$, a branched cover of degree $d$ is a non-constant holomorphic map $f:M\rightarrow N$ such that away from finitely many points $f$ is a degree $d$ topological covering map. Near each ramification point $q\in M$ there are local holomorphic coordinates such that $f$ can be written locally as $f(z)=z^{e_{q}}$, where $e_{q}\in \NN$ is the local ramification index. 

The Riemann--Hurwitz theorem gives the connection between the degree and the ramification indices: 
\begin{equation}\label{e:rh}
\chi(M)=d\cdot \chi(N)-\sum_{q\in M} (e_{q}-1).
\end{equation}
These ramification indices can be further grouped into the following data. Let $p_{1},\dots p_{n} \in N$ be the branching points, which are the images of the ramification points under $f$. For each $p_{i}$, its pre-image under $f$ contains finitely many points $q_{i}^{1}, \dots, q_{i}^{\ell_{i}}\in M$, and near each $q_{i}^{j}$ the map $f$ can be written as $f(z)=z^{\Pi_{i}^{j}}$ where $\Pi_{i}^{j}\in \NN$ denotes the {\em local ramification index}. Since $p_{i}$ is a branching point, at least one of the $\Pi_{i}^{j}$'s should be greater than 1. By the topological degree counting argument, we have 
\begin{equation}\label{e:piSum}
\sum_{j=1}^{\ell_{i}}\Pi_{i}^{j}=d, \ \forall i=1,\dots, n.
\end{equation}
It is easy to see that the set $\{q_{i}^{j}: i=1,\dots, n, \ j=1, \dots, \ell_{i}\} \subset N$ is exactly the set of ramification points. So the Riemann--Hurwitz formula~\eqref{e:rh} can be rewritten as
\begin{equation}\label{e:rh2}
\chi(M)=d\cdot \chi(N)-\sum_{i=1}^{n} \sum_{j=1}^{\ell_{i}} (\Pi_{i}^{j}-1).
\end{equation}
 
\begin{figure}[h]
\includegraphics[width=\textwidth]{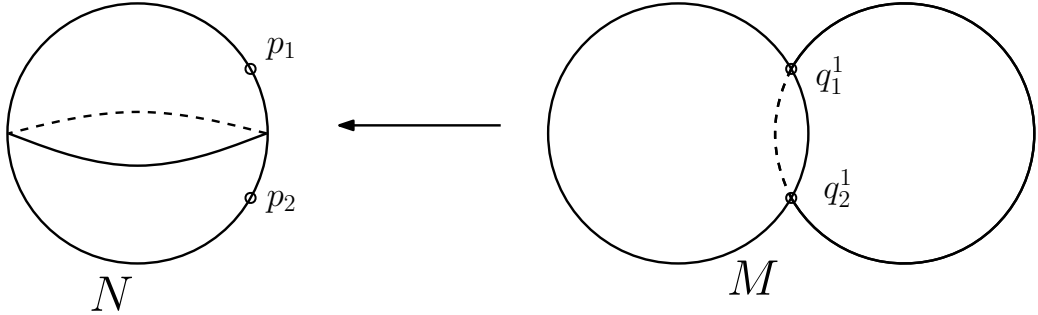}
 \caption{An example of a degree-2 branched cover $\bS^{2}\rightarrow \bS^{2}$, with branched data $\Pi=[(2),(2)]$.}
 \label{f:doublecover}
\end{figure}
 
The famous \textbf{Hurwitz existence problem} asks the following question:
\begin{question}
Given two Riemann surfaces $M$ and $N$, and the following set of data: 
\begin{itemize}
\item degree $d\in \NN$;
\item $\{\Pi=[(\Pi_{1}^{1}, \dots, \Pi_{1}^{\ell_{1}}), \ \dots, \ (\Pi_{i}^{1}, \dots, \Pi_{i}^{\ell_{i}}), \ \dots, \ (\Pi_{n}^{1}, \dots, \Pi_{n}^{\ell_{n}})]\in \NN^{\sum_{i=1}^{n}\ell_{i}}\}$ such that for each $i$, there exists at least one $j$ such that $\Pi_{i}^{j}\geq 2$,
\end{itemize}
where $(d, \Pi)$ 
satisfy the conditions~\eqref{e:piSum} and~\eqref{e:rh2}, does there exist a branched cover $f: M\rightarrow N$ that realizes this data?
\end{question}
This question dates back to Hurwitz~\cite{Hurwitz}, and has seen a lot of interesting developments. In the case when $\chi(N)\leq 0$, it is understood that all data is realizable, by the work of Edmonds--Kulkarni--Stong~\cite{EKS}. However, when $N$ is $\bS^{2}$, this problem is much more complicated. It is well known that there exists \textbf{exceptional} data $(d, \Pi)$ which is not realizable by a branched cover, but so far it is still a widely open problem to characterize all such data. Therefore, it is interesting to find new types of exceptional data, in the hope of giving more insight for finding a universal criterion.

In this work, we concentrate on the case when both $M$ and $N$ are $\bS^{2}$ and construct new types of exceptional data via the connection to spherical conical metrics. We give the following criterion which uses the notion of {\em admissible cone angle set} introduced in Definition~\ref{d:angle} below.
\begin{theorem}\label{t:main}
Let $(d, \Pi)$ be a set of data satisfying~\eqref{e:piSum} and~\eqref{e:rh2}. Assume there exists an admissible cone angle set $\vec \beta=(\beta_{1}, \dots, \beta_{n})\in (0,\infty)^{n}$ such that the associated vector 
$$\vec \beta \cdot \Pi:=(\Pi_{1}^{1}\beta_{1}, \ \dots, \ \Pi_{1}^{j}\beta_{1},\ \dots,\ \Pi_{i}^{j}\beta_{i},\ \dots,\ \Pi_{n}^{\ell_{n}}\beta_{n})\in (0,\infty)^{\sum_{i=1}^{n} \ell_{i}}$$
is \textbf{not} admissible by Definition~\ref{d:angle}. Then $(d, \Pi)$ is exceptional, i.e., it cannot be realized by a branched cover of $\bS^{2}$ over $\bS^{2}$.
\end{theorem}
 
As an application of the result above, we provide three kinds of exceptional data. We remark that Proposition~\ref{p:2k} is a special case of~\cite[Theorem~1.5]{PP1} and Proposition~\ref{p:rk} is a special case of~\cite[Theorem~1.4]{PP1}, with the results of~\cite{PP1} obtained using a different approach than the one used here. Proposition~\ref{p:3k} appears to be a new example of exceptional data.
\begin{proposition}\label{p:2k}
If $d=2k$ for $k\geq 2$, then the following set $(d, \Pi)$ is exceptional:
\begin{itemize}
\item $\Pi=[(k_{1},k_{2}), \ \underbrace{(2, \dots, 2)}_{k}, \ \underbrace{(2, \dots, 2)}_{k}]$, where $k_{1}+k_{2}=2k$ and $k_{1}\neq k_{2}$.
\end{itemize}
If in addition $k\geq 3$, then we also have the following exceptional $(d, \Pi)$:
\begin{itemize}	
\item $\Pi=[\underbrace{(2, \dots, 2)}_{k}, \ (\underbrace{2, \dots, 2}_{j_{1}}, 2k-2j_{1}), \ (\underbrace{2, \dots, 2}_{j_{2}}, 2k-2j_{2})]$, where $j_{1}+j_{2}=k$, $j_{1}\neq j_{2}$, and $j_{1}, j_{2}\geq 1$.
\end{itemize}
\end{proposition}

\begin{proposition}\label{p:3k}
When $d=3k$ where $k$ is odd and $k\geq 3$, the following set $(d, \Pi)$ is exceptional:
\begin{itemize}
\item $\Pi=[(k-2, \underbrace{2, \dots, 2}_{k+1}), \ (\underbrace{3, \dots, 3}_{k}), \ (\underbrace{3, \dots, 3}_{k})]$.
\end{itemize}
\end{proposition}

\begin{proposition}\label{p:rk}
When $d=rk$ where $r\geq 2$ and $k\geq 2$, the following two sets $(d, \Pi)$ are exceptional:
\begin{itemize}
\item $\Pi=[(2k-1, \underbrace{1, \dots, 1}_{(r-2)k+1}), \ (\underbrace{r, \dots, r}_{k}), \ (\underbrace{r, \dots, r}_{k})]$.
\item $\Pi=[(j_{1},\ j_{2},\ \underbrace{1, \dots, 1}_{(r-2)k}), \ (\underbrace{r, \dots, r}_{k}), \ (\underbrace{r, \dots, r}_{k})]$, where $j_{1}\neq j_{2}$ and $j_{1}+j_{2}=2k$. 
\end{itemize}
\end{proposition}

From the exceptional data constructed above, we get the following result regarding non-prime degrees, which was proved in~\cite{EKS} using a different method.
\begin{corollary}\label{c:nprime}
For every $d$ that is not a prime, there exists at least one set of data $(d, \Pi)$ that is exceptional.
\end{corollary}

The criterion in this paper is obtained through the connection with spherical conical metrics on $\bS^{2}$.
A spherical conical metric on a Riemann surface $M$ is a constant curvature one (``spherical'') metric, smoothly defined except finitely many cone points. Near any cone point the metric is asymptotic to a cone of certain angle. When the underlying manifold $M$ is $\bS^{2}$, there has been a lot of recent developments regarding the existence and uniqueness of such metrics. In particular, the cone angles need to satisfy certain conditions in order for such metrics to exist~\cite{Dey, Ere1, Ka, MP}. In~\S\ref{s:sph} we review those results and discuss the existence of spherical conical metrics depending on the conical data. Then in~\S\ref{s:cover} we connect the cone angle set of such metrics to branched covers of $\bS^{2}$, and obtain the testing criterion described in the main theorem.

In the search of exceptional branch data, there have been a lot of different approaches, see for example~\cite{PP3} for a review of available results and techniques. 
Many exceptions are given in various cases, see the classical results~\cite{EKS, Ger, Huse, KZ, Med1, Med2, OP} and more recent developments~\cite{Bar, MSS, Pa, PaP, PP1, PP2, SongXu, Zheng}. In~\cite{Zheng} all exceptional candidate branched covers with $n=3$ and $d\leq 10$ in the case of $M=N=\bS^{2}$ have been determined by computer. A variety of techniques have been used to find some new examples of exceptional data for arbitrarily large degrees~\cite{PaP, PP1, PP2, SongXu}. 
However the general pattern of what kind of data is realizable still remains unclear. The following conjecture suggesting connections with number-theoretic facts was proposed in~\cite{EKS} and is supported by strong evidence in~\cite{PaP}:
\begin{conjecture}[Prime degree conjecture]
If $(d, \Pi)$ is a set of branched cover data for $\bS^{2}\rightarrow \bS^{2}$ that satisfies~\eqref{e:piSum} and~\eqref{e:rh2}  and the degree $d$ is a prime number, then this set of data is realizable.
\end{conjecture}
To characterize the necessary and sufficient conditions for all the possible branched cover data so far still is an open problem. In this paper we discuss the connection of this old problem to a new field, which uncovers new examples of exceptional data and in addition provides more evidence of the prime degree conjecture.

\bigskip

\noindent\textbf{Acknowledgement}: The author would like to thank Xiaowei Wang for suggesting this topic, and Rafe Mazzeo and Alex Wright for very useful discussions. The author is grateful to Carlo Petronio for pointing her to~\cite[Theorems~1.4 and~1.5]{PP1},
see the remark preceding Proposition~\ref{p:2k}. The author would also like to thank the referee for the careful reading and valuable comments.

\section{Admissible cone angle data for spherical conical metrics}\label{s:sph}
Spherical metrics with conical singularities are defined by the following conical data: given $n$ distinct points $\frakp=(p_{1}, \dots, p_{n}), \ p_{i}\in M$ and cone angle data $\vec \beta=(\beta_{1}, \dots, \beta_{n}),\ \beta_{i}>0$, $g$ is a smooth metric on $M\setminus \frakp$ with constant curvature one, and near each puncture $p_{i}$ the metric is asymptotically conical with angle $2\pi \beta_{i}$, that is, there exist local coordinates near $p_{i}$ such that the metric is given by
\begin{equation}\label{e:conformal}
e^{\phi(z)} |z|^{2(\beta_{i}-1)}|dz|^{2}
\end{equation} 
with $\phi(z)$ being a smooth function. Locally near $p_{i}$ there is also a geodesic coordinate description given by 
\[
g(r, \theta)=dr^{2}+\beta_{i}^{2}\sin^{2}r \,d\theta^{2}.
\]
Finding constant curvature metrics with a given conformal structure and conical data can be seen as  a singular uniformization problem, which has a long history and has been very active recently. The problem of existence and uniqueness of such metrics given the conical data $(M, \frakp,\vec \beta)$ is still open when some of the cone angles are bigger than $2\pi$. It is obvious that the angles need to satisfy the Gauss--Bonnet constraint
\begin{equation}\label{e:gb}
\chi(M) + \sum_{j=1}^n (\beta_j-1) >0
\end{equation}
because the quantity on the left-hand side equals $\frac{1}{2\pi} A$ where $A$ is the area of the metric $g$. However, it turns out that there are other constraints on~$\vec \beta$ in order for such metrics to exist. 
The recent breakthrough by Mondello--Panov~\cite{MP}, followed by works of Dey~\cite{Dey}, Kapovich~\cite{Ka}, and Eremenko~\cite{Ere1}, establishes necessary and sufficient conditions on 
$\vec{\beta}$ for which there exists such a spherical metric  on $\bS^2$ with these prescribed cone angle 
parameters. We summarize those results below. 

\begin{definition}[Admissible cone angles] 
Suppose $n\geq 2$. Let $d_{\ell^{1}}(\vec x, \vec y):=\sum_{i=1}^{n}|x_{i}-y_{i}|$ be the $\ell^{1}$ distance on $\RR^{n}$, and $\ZZ^{n}_{odd}$ be the set of integer points $(z_{1}, \dots, z_{n})\in \ZZ^{n}$ such that $\sum_{i=1}^{n} z_{i}$ is odd.
Let $\vec \beta=(\beta_{1}, \dots, \beta_{n})\in ((0,\infty)\setminus\{1\})^{n}$ be an $n$-tuple satisfying~\eqref{e:gb}. We say $\vec \beta$ is an {\em admissible cone angle set} if one of the following conditions holds:
\begin{enumerate}
\item [(a)] $d_{\ell^{1}}(\vec \beta-\vec 1, \ZZ_{odd}^{n})>1$;
\item [(b)] $n=2$ and $\beta_{1}=\beta_{2}\in (0,\infty)\setminus \NN$.
\item [(c)] $d_{\ell^{1}}(\vec \beta-\vec 1, \ZZ_{odd}^{n})=1$, and $\beta_{i}\in \NN$ for all $i$. Moreover, $2\max_{i}(\beta_{i}-1)\leq \sum_{i=1}^{n} (\beta_{i}-1)$.
\item [(d)] $d_{\ell^{1}}(\vec \beta-\vec 1, \ZZ^{n}_{odd})=1$, and (up to reordering) there exists $1< m <n$ such that $\beta_{1}, \dots \beta_{m}\notin \NN$, $\beta_{m+1}, \dots, \beta_{n}\in \NN$. Moreover $\vec \beta$ satisfies the following ``coaxial conditions'':
\begin{itemize}
\item There exists $\{\epsilon_{i}\}_{i=1}^{m}$ with $\epsilon_{i}\in\{\pm 1\}$ such that 
$$k'=\sum_{i=1}^{m}\epsilon_{i}\beta_{i}\geq 0.$$
\item  $k''=\sum_{i=m+1}^{n} \beta_{i}-n-k'+2\geq 0$ and $k''$ is even. 
\item Let $(\beta_{1}, \dots, \beta_{m}, \underbrace{1,\dots, 1}_{k'+k''})=\eta(b_{1}, \dots, b_{m+k'+k''})$ where $b_{i}$'s are integers whose greatest common divisor is 1, then 
$$2\max_{i=m+1}^{n}\beta_{i}\leq \sum_{i=1}^{m+k'+k''} b_{i}.$$
\end{itemize}
\end{enumerate}\label{d:angle}
\end{definition}
Note that by this definition, if $d_{\ell^{1}}(\vec \beta-\vec 1, \ZZ^{n}_{odd})<1$, or if $d_{\ell^{1}}(\vec \beta-\vec 1, \ZZ^{n}_{odd})=1$ but $\beta_{i}\notin \NN$ for all $i$ when $n\geq 3$, then this cone angle set $\vec\beta$ is automatically not admissible. 

\begin{figure}[h]
\centering
\includegraphics[width=0.6\textwidth]{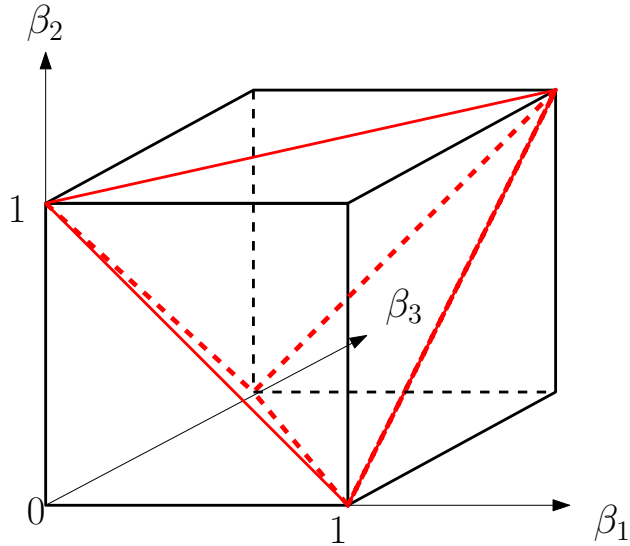}
 \caption{An illustration of the set of admissible cone angles when $n=3$ and all $\beta_{i}\leq 1$: in this case the admissible region is the interior of the tetrahedron bounded by red edges, plus one point $(1,1,1)$ (which is the trivial smooth case). None of the boundary points of this tetrahedron is admissible by Definition~\ref{d:angle}.}
 \label{f:MP}
\end{figure}

We also note here that, when $\beta_{i}=1$, the cone angle is $2\pi$, which gives a smooth point of the metric. And in the above criterion we assume none of the $\beta_{i}$'s is equal to 1. However, in Theorem~\ref{t:main}, $\vec \beta$ and $\vec \beta \cdot \Pi$ are allowed to have some of the entries to be equal to 1. Therefore we need to first remove all the entries that are equal to 1, and then apply the admissible angle criterion to the ``shortened'' cone angle vector. For example, $\vec \beta=(1/2, 1/2, 1)$ is admissible according to situation (b) in the definition above.

The definition above gives a complete description of possible cone angle combinations for spherical conical metrics on $\bS^{2}$, which is summarized below.
\begin{proposition}
If $M=\bS^{2}$, then there exists a spherical conical metric with cone angle data $\vec\beta$ if and only if $\vec \beta$ is an admissible cone angle set according to Definition~\ref{d:angle}.
\end{proposition}
\begin{proof}
By the work of Mondello--Panov~\cite{MP}, if such metric exists then the ``holonomy condition''
\begin{equation}\label{e:holo}
d_{\ell^{1}}(\vec \beta-\vec 1, \ZZ_{odd}^{n})\geq 1
\end{equation}
must hold. And if the strict inequality holds, which is situation (a), then there exists at least one such metric. 

By the work of Dey~\cite{Dey}, if the equality in~\eqref{e:holo} holds, and $\beta_{i}\notin\NN$ for any $i$, then the only admissible set is when $n=2$ and $\beta_{1}=\beta_{2}$ which is situation (b). 

So we are left with the case when the equality in~\eqref{e:holo} holds and at least some of the $\beta_{i}$'s are integers. When they are all integers, which is situation (c), Kapovich's result~\cite{Ka} gives the sufficient and necessary condition. And the remaining case, which is situation (d), is given by Eremenko~\cite{Ere1}.
\end{proof}

One direct consequence from the above theorem is that, when $n=2$, the admissible condition is very restrictive. Combining situation (b) and (c) we have
\begin{corollary}\label{c:football}
When $n=2$, the necessary and sufficient condition for the existence of a spherical conical metric is
$\beta_{1}=\beta_{2}$. 
\end{corollary}
Depending on whether $\beta_{i}$ is an integer or not~\cite{Tr2}, the spherical metrics have different behavior which can be seen from Figure~\ref{f:football}. 

\begin{figure}[h]
\includegraphics[width=\textwidth]{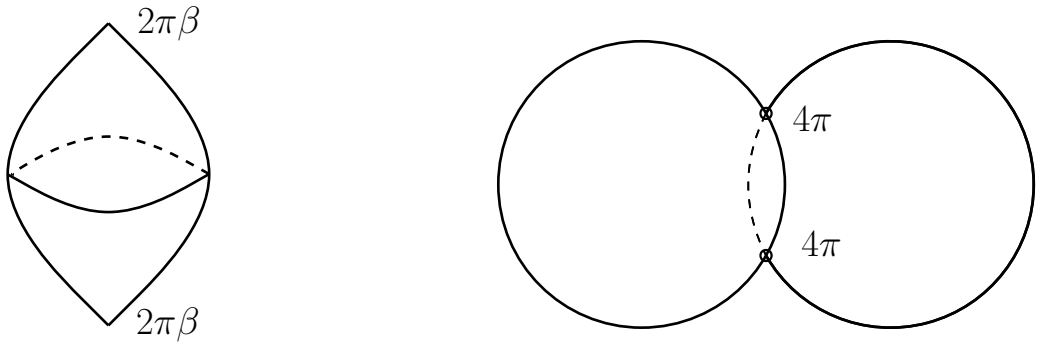}
 \caption{The spherical conic metrics with two cone points are given by two types: (a) When $\beta_{1}=\beta_{2}\notin \NN$, they are footballs where two cone points are restricted to be antipodal (on the left); (b) When $\beta_{1}=\beta_{2}=\beta\in \NN$, they are obtained by degree-$\beta$ covers of $\bS^{2}$ where the two cone points are free to move around (on the right).}
 \label{f:football}
\end{figure}

When all the cone angles are less than $2\pi$, the above result can be summarized by a previous result due to Troyanov~\cite{Tr, Tr2}, which is easier to check for the later construction of $\vec\beta$ in Theorem~\ref{t:main}.
\begin{proposition}\label{p:Tro}
If $\beta_{i}\in (0,1)$ for all $i=1, \dots, n$, then $\vec \beta$ gives rise to a spherical conical metric on $\bS^{2}$ if and only if one of the following is true:
\begin{itemize}
\item $n=2$, and $\beta_{1}=\beta_{2}\in (0,1)$;
\item $n\geq 3$, $\sum_{i=1}^{n}(\beta_{i}-1)>-2$, and 
\begin{equation}
\min\{2,2\beta_j\} + n- 2> \sum_{i=1}^{n} \beta_i, \qquad j = 1, \ldots, n.
\end{equation}
\end{itemize}
\end{proposition}

Another type of result we will need is the existence of conical metrics depending on the position of cone points. When the number of cone points (which is denoted by $n$ in our notation) is less than 4, there is no restriction. This is because we can find a Mobius transform $u: \bS^{2}\rightarrow \bS^{2}$ such that any given three distinct points on $\bS^{2}$ are mapped to three other given points. However, when $n\geq 4$, different sets of branching points correspond to different conformal classes and we no longer have such Mobius transforms. To resolve this problem, we have the following result on the freedom of cone point positions, due to Troyanov~\cite{Tr} and Mazzeo--Weiss~\cite{MW}:
\begin{proposition}\label{p:position}
If $\beta_{i}\in (0,1)$ for all $i=1, \dots, n$ and $\vec \beta$ is an admissible angle set, then for any $n$ distinct points $\{p_{1}, \dots, p_{n}\}$ on $\bS^{2}$, there exists a spherical conical metric with cone angle $2\pi\beta_{i}$ at point $p_{i}$.
\end{proposition}
\begin{proof}
See Theorem C of~\cite{Tr} and Theorem 2 of~\cite{MW}.
\end{proof}

However such a result does not hold when at least one of the cone angles is bigger than $2\pi$. The ongoing work by Mazzeo and the author~\cite{MZ, MZ2} suggests that for a given set of admissible $\vec\beta$, in some cases there are restrictions on the cone point positions.
Luckily, to remedy this problem so that in Theorem~\ref{t:main} we can use admissible cone angles $\beta_{i}$ for any positive $\beta_{i}$, we have the following realizability of branched covers regarding the position of branching points:
\begin{proposition}\label{p:Riemann}
Given a set of branching data $(d,\Pi)$, if there exists a corresponding branched cover $f:\bS^{2}\rightarrow \bS^{2}$ with $n$ branching points, then for any distinct $n$ points $\{p_{1}, \dots, p_{n}\}$ on $\bS^{2}$, there exists a branched cover $\tilde f$ with branching points $\{p_{i}\}_{i=1}^{n}$ and the same branching data $(d, \Pi)$.
\end{proposition}
\begin{proof}
This is essentially the Riemann's existence theorem regarding extending a topological cover to a branched cover of Riemann surfaces. See for example~\cite[Chapter 7 Theorems 3(b) and 4]{Nar} for the proof.
\end{proof}

\section{Exceptional data for branched covers of $\bS^{2}$}\label{s:cover}
We now prove the main theorem.
\begin{proof}[Proof of Theorem~\ref{t:main}]
For a Riemann surface $M$, suppose a branched cover $f: M\rightarrow \bS^{2}$ corresponding to $(d, \Pi)$ exists, with $\frakp:=(p_{1}, \dots, p_{n})\subset \bS^{2}$ being the set of branching points. And suppose $g$ is a spherical conical metric on $\bS^{2}$ such that its cone points are contained in $(p_{1}, \dots, p_{n})$. Let $\vec \beta=(\beta_{1}, \dots, \beta_{n}) \in (0,\infty)^{n}$ be the cone angle data associated to those points, and notice that it is possible that the branching points are actually smooth points of $g$, i.e., we allow some or all of the $\beta_{i}$'s to be equal to 1.

First we show that via pullback, $g$ lifts to a spherical conical metric $f^{*}g$ on $M$, with cone angle data $\vec \beta\cdot \Pi$. To see this, we first look at the topological degree-$d$ cover obtained by removing all the branching points on $\bS^{2}$ and all the ramification points on $M$, denoted by $M\setminus f^{-1}(\frakp) \rightarrow \bS^{2}\setminus \frakp$. Since the metric $g$ is smooth with curvature one at any point on $\bS^{2}\setminus \frakp$, and $f$ is a local diffeomorphism restricted to this cover, we can see that $f^{*}g$ is smooth and with curvature one at any point when restricted to $M\setminus f^{-1}(\frakp)$. On the other hand, for any branching point $p_{i}\in \bS^{2}$ and one of its pre-images $q_{i}^{j}\in M$, there exists a local coordinate $z$ near $q_{i}^{j}$ such that $f$ is written as $f(z)=z^{\Pi_{i}^{j}}$, where $\Pi_{i}^{j}\in \NN$ is the ramification index. Since near $p_{i}$ the metric $g$ is given by
$$
g=e^{\phi(z)}|z|^{2(\beta_{i}-1)}|dz|^{2}, 
$$ 
then by computing the pullback we get that near $q_{i}^{j}$,
$$
f^{*}g=(\Pi_{i}^{j})^{2} e^{\phi(z^{\Pi_{i}^{j}})}|z|^{2(\Pi_{i}^{j}\beta_{i}-1)}|dz|^{2}
$$
which is a cone point with cone angle $2\pi(\Pi_{i}^{j}\beta_{i})$.

Now we prove the theorem by contradiction. Suppose $f: \bS^{2}\rightarrow \bS^{2}$ is a branched cover corresponding to $(d, \Pi)$, such that $(p_{1}, \dots, p_{n})$ are the branching points on $\bS^{2}$. And suppose we have the admissible cone angle set $\beta_{i}\in (0,\infty)^{n}$. By Proposition~\ref{p:position}, if all the cone angles are less than $2\pi$, there exists such a spherical conical metric $g$ with cone angle $2\pi \beta_{i}$ at point $p_{i}, \ i=1, \dots, n$. Otherwise we use Proposition~\ref{p:Riemann} and another branched cover $\tilde f$ with the same branching data, so that the branching points of $\tilde f$ match with the cone points of $g$. Then by the discussion above, the pullback metric $f^{*}g$ is a spherical conical metric on $\bS^{2}$ with cone angle $2\pi(\Pi_{i}^{j}\beta_{i})$ at point $q_{i}^{j}, \ i=1, \dots, n, \ j=1, \dots, \ell_{i}$. However, 
$\vec \beta\cdot \Pi$ is not admissible, so there does not exist such a metric on $\bS^{2}$, which is a contradiction. So we have proved that such branched cover $f$ does not exist.
\end{proof}
\begin{remark}
We note here that, from the discussion before Proposition~\ref{p:position}, we actually do not need to worry about the position of the branching points if there are at most three branching points (which is the case for all the examples below). The reason is that in the proof we can use an appropriate Mobius transform $u$, such that $u\circ f$ is a branched cover with matching branching points to the given conical metric $g$.
\end{remark}

We now apply the theorem to obtain new types of exceptional branching data, which is done by carefully picking the base data $\vec \beta$ for different cases.
\begin{proof}[Proof of Proposition~\ref{p:2k}]
For the first case, we take $\vec \beta=(\beta_{1},\beta_{2},\beta_{3})=(\frac{1}{2},\ \frac{1}{2}, \ \frac{1}{2})$, which is admissible by Proposition~\ref{p:Tro}. However, the angle data obtained by covering is given by 
$$
\vec\beta\cdot \Pi=\bigg(\frac{k_{1}}{2},\ \frac{k_{2}}{2},\ \underbrace{1, \dots, 1}_{2k}\bigg)
$$
which after removing all the smooth points (i.e. removing all the 1's) we get a metric with two distinct cone angles $2\pi(\frac{k_{1}}{2})\neq 2\pi (\frac{k_{2}}{2})$. By Corollary~\ref{c:football}, such a cone angle set is not admissible. 

Similarly for the second case, we use the same $\vec \beta=(\frac{1}{2},\ \frac{1}{2}, \ \frac{1}{2})$, then 
$$
\vec\beta\cdot \Pi=(\underbrace{1,\dots, 1}_{k+j_{1}}, \ k-j_{1}, \ \underbrace{1,\dots, 1}_{j_{2}}, \ k-j_{2}).
$$
which is shortened to $(k-j_{1},k-j_{2})$. With $j_{1}\neq j_{2}$, again by Corollary~\ref{c:football}, such a cone angle set is not admissible. 

Therefore in both cases, by applying Theorem~\ref{t:main} we get that such data $(d, \Pi)$ is not realizable.
\end{proof}

\begin{proof}[Proof of Proposition~\ref{p:3k}]
Take $\vec \beta=(\frac{1}{2}, \frac{2}{3}, \frac{2}{3})$ which is admissible by Proposition~\ref{p:Tro}. However, the lifted angle data on the branched cover is given by
$$
\vec\beta\cdot \Pi=\bigg(\frac{k-2}{2},\ \underbrace{1, \dots, 1}_{k+1},\ \underbrace{2, \dots, 2}_{2k}\bigg)
$$ 
which, after removing all the 1's, is shortened to 
$$
\vec \alpha=\bigg(\frac{k-2}{2},\ \underbrace{2, \dots, 2}_{2k}\bigg).
$$
Since we assume that $k$ is odd and $k\geq 3$, there exists an odd integer $z$ such that 
$$
\bigg|z-\bigg(\frac{k-2}{2}-1\bigg)\bigg|=\frac{1}{2}.$$ 
Then
the $\ell^{1}$ distance of $\vec \alpha-\vec 1$ and the odd integer point 
$$\vec z:=(z, \ \underbrace{1, \dots, 1}_{2k})$$ 
is given by
$
d_{\ell^{1}}(\vec \alpha-\vec 1, \vec z)=\frac{1}{2}<1.
$
So from Definition~\ref{d:angle}, $\vec\beta\cdot \Pi$ is not admissible for conical metrics.
\end{proof}

\begin{proof}[Proof of Proposition~\ref{p:rk}]
Take $\vec \beta=(1, \frac{1}{r}, \frac{1}{r})$, which is admissible by Proposition~\ref{p:Tro} (in fact it corresponds to a football with antipodal angle $2\pi\frac{1}{r}$). Compute the lifted angle data and we get
$$
\vec\beta\cdot \Pi=(2k-1,\ \underbrace{1, \dots, 1}_{rk+1}).
$$
After removing all the 1's and using that $k\geq 2$, we get
$$
\vec \alpha:=(2k-1)
$$
which is an integer vector of only one entry. Since $k\geq 2$, we obtain a spherical metric with one single angle $2\pi(2k-1)$. By looking at situation (c) in Definition~\ref{d:angle}, such $\vec \alpha$ is not admissible. Hence we get that such vector $\vec\beta\cdot \Pi$ is not admissible. 

Similarly for the second data, if we take $\vec \beta=(1, \frac{1}{r}, \frac{1}{r})$, then the lifted angle data is
$$
\vec \beta \cdot \Pi=(j_{1}, \ j_{2}, \ \underbrace{1, \dots, 1}_{rk}),
$$
which is shortened to
$$
(j_{1}, \ j_{2}), j_{1}\neq j_{2}.
$$
Hence by Corollary~\ref{c:football}, such a cone angle set is not admissible.
\end{proof}

If $d$ is not a prime, then $d$ admits at least one of decompositions: $d=2k$ or $d=rk$ for some odd $r$ with $k\geq 2$. So combining Proposition~\ref{p:2k} and Proposition~\ref{p:rk}, we immediately see that every such non-prime degree $d$ admits at least one exceptional branching data set. Hence we have Corollary~\ref{c:nprime}.

\end{document}